\documentclass[12pt,french,english]{article}
\usepackage{lipsum}

\usepackage[nottoc]{tocbibind}
\usepackage[utf8]{inputenc}
\usepackage[english]{babel}
\usepackage{amsmath}
\usepackage{amsfonts}
\usepackage{amssymb}
\usepackage{amsthm}
\usepackage[normalem]{ulem}
\usepackage{mathtools}
\usepackage{appendix}
\usepackage{tabu}
\usepackage[dvipsnames]{xcolor}
\usepackage{bm}
\usepackage{hyperref}
\usepackage[all,cmtip]{xy}

\usepackage[initials]{amsrefs}
\usepackage{listings}
\usepackage{tikz-cd}
\usepackage[mathscr]{euscript}
\usepackage[left=2.5cm,right=2.5cm,top=2.5cm,bottom=3cm]{geometry}
\lstdefinelanguage{Macaulay2}{
comment=[l]{--},
alsoletter={'},
alsoother={_},
}
\theoremstyle{plain}
\newtheorem{Theorem}{Theorem}[section]
\newtheorem{prop}[Theorem]{Proposition}
\newtheorem{conj}[Theorem]{Conjecture}
\newtheorem{thm}[Theorem]{Theorem}
\newtheorem{cor}[Theorem]{Corollary}
\newtheorem{lem}[Theorem]{Lemma}

\newtheorem*{Ackn}{Acknowledgments}

\theoremstyle{definition}

\newtheorem{dfn}[Theorem]{Definition}

\theoremstyle{remark}

\newtheorem{rmk}[Theorem]{Remark}

\DeclareMathOperator{\Q}{\mathbb Q}

\DeclareMathOperator{\Co}{\mathbb C}

\DeclareMathOperator{\ns}{NS}

\DeclareMathOperator{\h}{H}

\DeclareMathOperator{\pr}{\mathbb{P}}

\DeclareMathOperator{\ob}{\mathcal{O}}

\DeclareMathOperator{\Bl}{Bl}

\makeatletter
\newcommand*{\rom}[1]{\expandafter\@slowromancap\romannumeral #1@}
\makeatother
\begin{document}
\author{Abugaliev Renat}
\title{Characteristic foliation on hypersurfaces with positive Beauville-Bogomolov-Fujiki square}
\maketitle
\begin{abstract}
Let $Y$ be a smooth hypersurface in a projective irreducible holomorphic symplectic manifold X of dimension 2n. The characteristic foliation $F$ is the kernel of the symplectic form restricted to Y. In this article we prove that a generic leaf of the characteristic foliation is dense in Y if Y has positive Beauville-Bogomolov-Fujiki square.
\end{abstract}
\section{Introduction}
First, we recall the definition of our main object of study, that is an irreducible holomorphic symplectic manifold.
\begin{dfn}
Let $X$ be a smooth projective variety over $\Co$. It is an irreducible holomorphic symplectic (IHS) if it satisfies the following properties:
\begin{itemize}
    \item $\h^0(X,\Omega^2_X)=\Co \sigma$, where $\sigma$ is a holomorphic symplectic form (at any point of $X$);
    \item $\h^1(X,\ob_X)=0$;
    \item $\pi_1(X)=0$.
\end{itemize}
\end{dfn}

 Note that a holomorphic symplectic form $\sigma$ on a smooth variety $X$ induces an isomorphism between the vector bundles $T_X$ and $\Omega_X$. Indeed, one can map a vector field $v$ to the differential form $\sigma(v,*)$.
\begin{dfn}Let $Y$ be a hypersurface in $X$ and $Y^{sm}$ smooth locus of $Y$. Consider the restriction of $T_X$ to $Y$. The restriction of a symplectic form to any codimension one subspace has rank $2n-2$ i.e. has one-dimensional kernel. The orthogonal complement of the bundle $T_{Y^{sm}}$ in  $T_{X|Y^{sm}}$ is a line subbundle $F$ of $T_{Y^{sm}}\subset T_{X} | _{Y^{sm}}$. We call the rank one subbundle $F\subset T_{Y^{sm}}$ the {\bf characteristic foliation}.
\end{dfn}
Assume $Y$ is smooth. Since $Y=Y^{sm}$, $F$ is a subbundle of $T_Y$. Furthermore, $F$ is isomorphic to the conormal bundle $\mathcal{N}_{Y/X}$ (which is isomorphic to $\ob_Y(-Y)$ by the adjunction formula). Indeed, consider the following short exact sequence: 
$$ 0\to T_Y \to T_X|Y \to \ob_Y(Y) \to 0.$$
Applying the isomorphism $T_X\cong \Omega_X$, we obtain that $F\cong\ob_Y(-Y)$.\\

  Our first question is whether it is algebraically integrable. Jun-Muk Hwang and Eckart Viehweg showed in \cite{HV} that if $Y$ is of general type, then $F$ is not algebraically integrable.  In paper \cite{AC} Ekaterina Amerik and Frédéric Campana  completed this result to the following.
\begin{thm}\cite[Theorem~1.3]{AC} \label{ac}
 Let $Y$ be a smooth hypersurface in an irreducible holomorphic symplectic manifold $X$ of dimension at least $4$. Then  the characteristic foliation on $Y$ is algebraically integrable if and only if $Y$ is uniruled, i.e. covered by rational curves.
\end{thm}

The next step is to ask what could be the dimension of the Zariski closure of a generic leaf of $F$. In dimension $4$ the situation is understood thanks to Theorem \ref{ag}.
\begin{thm}[\cite{AG}]\label{ag}
Let $X$ be an irreducible holomorphic symplectic fourfold and let $Y$ be an irreducible smooth hypersurface in $X$. Suppose that the characteristic foliation $F$ on $Y$ is not algebraically integrable, but there exists a meromorphic fibration on $p:Y \dasharrow C$  by surfaces invariant under $F$ (see Definition \ref{inv}). Then  there exists a rational Lagrangian fibration $X \dasharrow B$ extending $p$. In particular, the Zariski closure of a generic leaf is an abelian surface.
\end{thm}
This leads to the following conjecture. 
\begin{conj}[Campana]\label{conja}
Let $Y$ be a smooth hypersurface in an irreducible holomorphic symplectic manifold $X$ and let $q$ be Beauville-Bogomolov form on $\h^2(X,\Q)$. Then:
\begin{enumerate}
\item If $q(Y,Y)>0$, a generic leaf of $F$ is Zariski dense in $Y$;
 \item If $q(Y,Y)=0$, the Zariski closure of a generic leaf of $F$ is an abelian variety of dimension $n$;
 \item If $q(Y,Y)<0$, $F$ is algebraically integrable and $Y$ is uniruled.
 \end{enumerate}
 \end{conj} 
Let us explain why this conjecture is plausible and formulate the results we achieved.\\

 \textbf{Case of $q(Y,Y)<0$.} In this case the conjecture is easy to prove. By \cite[Theorem~ 4.2 and Proposition~4.7]{Bo} $Y$ is uniruled, if $q(Y,Y)<0$. There is a dominant rational map $f:Y\dasharrow W$, such that the fibers of $f$ are rationally connected (see \cite{cam} and \cite[Chapter ~IV.5]{KolRC}). Rationally connected varieties do not have non-zero holomorphic differential forms. Thus, the form $\sigma|_Y$ is the pull-back of some form $\omega \in \h^0(W, \Omega^2_W)$ and for any point $x$ in $Y$ the relative tangent space $T_{Y/W,x}$ is  the kernel of the form $\sigma|_Y$. The kernel of $\sigma|_Y$ has dimension one at every point. So, the rational map $f:Y\dasharrow W$ is a fibration in rational curves and these rational curves are the leaves of the foliation $F$.\\

\textbf{Case of $q(Y,Y)=0$}. Conjecturally, $X$ admits a rational lagrangian fibration, and hypersurface $Y$ is  the inverse image of a hypersurface of its base (this conjecture was proved for manifolds of K3 type in \cite{BM} and for manifolds of Kummer type \cite{yos}, for O'Grady 6 type in \cite{MR} and for O'Grady 10 type in \cite{MO}). Moreover a rational  Lagrangian fibration can be replaced with a regular Lagrangian fibration, if we assume that the divisor $Y$ is numerically effective. In the previous work \cite{Ab1} we consider an irreducible holomorphic symplectic manifold $X$ equipped with a regular Lagrangian fibration $\pi:X\to B$. The main result of \cite{Ab1} is the following.
\begin{thm}\label{t2}Let $X$ be a projective irreducible holomorphic symplectic manifold and \\$\pi:X \to B$ a Lagrangian fibration. Consider a hypersurface $D$ in $B$ such that its preimage $Y$ is a smooth irreducible hypersurface in $X$. Then the closure of a generic leaf of the characteristic foliation on $Y$ is a fiber of $\pi$ (hence an abelian variety of dimension $n$). 
\end{thm}
It proves the second case of conjecture \ref{conja} for a nef hypersurface assuming the Lagrangian conjecture. \\

\textbf{Case of $q(Y,Y)>0$}. In this paper we prove the last case of conjecture \ref{conja} and also remark that the second case holds for a non-nef hypersurface. Let us write up the structure of the article and state the main theorems. First we prove the conjecture \ref{conja} for a nef and big hypersurface.    
 \begin{thm}\label{tnef}
Let $Y$ be a smooth nef and big hypersurface in an irreducible holomorphic symplectic manifold $X$. Then a generic leaf of the characteristic foliation on $Y$ is Zariski dense in the hypersurface $Y$. 
\end{thm}
We begin with proving this theorem for an ample hypersurface in section \ref{secample}. To apply the methods from section \ref{secample} for a nef and big divisor we need the the Lefschetz hyperplane theorem (LHT). The LHT is well known for a smooth ample hypersurface, but in general is false for a nef and big hypersurface in an arbitrary ambient variety. However, the LHT appears to be true for a smooth nef and big hypersurface in an irreducible holomorphic symplectic manifold. Let us formulate the LHT for a nef and big hypersurface in an irreducible holomorphic symplectic manifold. This theorem also might be interesting out of context of the characteristic foliation.
\begin{thm}\label{lht}
Let $Y$ be a smooth, nef and big hypersurface in an irreducible holomorphic symplectic manifold $X$. Then for $i<\dim X$ the restriction induces an isomorphism on the cohomology groups $\h^i(X,\mathbb{Q})\cong \h^i(Y,\mathbb{Q})$.
\end{thm}
 We prove it in section \ref{seclht}. In section \ref{secnef} we adapt the arguments of section \ref{secample} to prove theorem \ref{tnef}.
 The following theorem has been indicated to the author by Jorge Vitorio Pereira. 
\begin{thm}\label{per}If $Y$ is smooth and uniruled then $q(Y,Y)<0$.
\end{thm}
The proof uses the characteristic foliation and shall be given in section \ref{nenef}. Notice that this theorem also provides the solution of the second case of conjecture \ref{conja} in the non-nef case, completing theorem \ref{t2}.
 \begin{Ackn}I am very grateful to Ekaterina Amerik for kind and patient mentorship, to Jorge Vitorio Pereira for an excellent idea. I also want to thank Alexey Gorinov, Emanuele Macri and Misha Verbitsky for helpful conversations. The study has been partially funded within the framework of the HSE University Basic Research Program
and the Russian Academic Excellence Project ’5-100’.
 \end{Ackn}

\section{The local deformations of holomorphic symplectic manifolds}\label{torellisec}
In this section we briefly recall the description of the local deformation space of a holomorphic symplectic manifold. The aim of the section is to show that a couple $(X,Y)$ of a holomorphic symplectic manifold $X$ and a divisor $Y$ with $q(Y,Y)>0$ deforms to a holomorphic symplectic manifold with an ample divisor. \\

Let us recall the local Torelli theorem. Let $X$ be an IHS manifold and $Def(X)$ the space of its local deformations. The manifold $Def(X)$ is smooth of dimension $b_2(X)-2$ by \cite{bog}. The local system $\h^2(X,\mathbb{Z})$ with the BBF form is constant over $Def(X)$, but the subspace $\h^{2,0}(X')$ varies with the deformation $X'$. Let $\sigma'\in \h^2(X',\mathbb{C})\cong \h^2(X,\mathbb{C})$ be a holomorphic symplectic form on $X'$, using this cohomology class we define the local period map $Per: Def(X)\to \pr(\h^2(X,\Co))$. 
$$Per:X'\mapsto \Co\cdot \sigma'.$$
It is easy to see from the properties of the polarization of a Hodge structure that, $q(\sigma',\overline{\sigma'})>0$ and $q(\sigma',\sigma')=0$. Thus the image of $Per$ is contained in the quadric $Q\subset \pr(\h^2(X,\mathbb{C}))$ defined by the BBF quadratic form. The Local Torelli theorem states that locally  it is an isomorphism.
\begin{thm}[The Local Torelli theorem] The period map
$$Per: B\to Q\subset \pr(\h^2(X,\Co))$$
is surjective and biholomorphic on a small open subset of $Q$.
\end{thm}
Next, we describe a relation between a deformation $X'$ and its image under the period map.
\begin{lem}\label{perdiv}Let $\alpha$ be a Hodge class in $\h^2(X,\mathbb{Z})$. Than it remains Hodge class on a deformation $X'\in B$ if and only if the class $\sigma'$ as above is orthogonal to $\alpha$. We denote this subspace of $Def(X)$ by $Def(X)^{\alpha}$.
\end{lem}
So the set of the deformations of $X$ preserving a divisor $D$ is a codimension one submanifold of $B$. 
\begin{prop}\cite[Theorem~2]{H1}\label{thuy} Let $X$ be a holomorphic symplectic manifold. Then $X$ is projective if and only there exists a divisor $Y$ with $q(Y,Y)>0$.
\end{prop}
\begin{cor}\label{cordef}Let $X$ be a holomorphic symplectic manifold and $Y$ a divisor with $q(Y,Y)>0$. Denote by $\alpha\in \h^2(X,\mathbb{Z})$ the class of $Y$. Then the class $\alpha$ is ample in a very general deformation $X'$ of $X$ preserving the Hodge class $\alpha$ (i.e. $[X'] \in Def(X)^{\alpha}$).   
\end{cor}
\begin{proof}
It follows from lemma \ref{perdiv} that the $\mathrm{NS}(X')$ is generated by $\alpha$. By proposition \ref{thuy} $X'$ is projective and hence $\alpha$ is ample on $X'$. 
\end{proof}
\section{Lefschetz's theorems}
For the proof of theorem \ref{lht} below we will use some specific versions of the Lefschetz's theorem about hyperplane section and hard  Lefschetz's theorem. Let us first recall the hard Lefschetz theorem.
\begin{thm}Let $X$ be a compact K\"{a}hler manifold of dimension $m$ (for example projective) and let $\omega \in \h^2(X,\Co)$ be a  K\"{a}hler class (an ample class). Define the Lefschetz operator $$L: \h^k(X,\Co)\to \h^{k+2}(X,\Co)$$ as the cup-product with the  K\"{a}hler class $\omega$. Then the operator $$L^{m-k}: \h^k(X,\Co) \xrightarrow{\sim} \h^{2m-k}(X,\Co)$$
is an isomorphism for every $k \leq m$.
\end{thm}
\begin{proof}
See \cite[Theorem~6.25]{Vo1}
\end{proof}
\begin{cor}\label{hard}Let $X$ a be holomorphic symplectic manifold and $Y$ a divisor with $q(Y,Y)>0$. Then the hard Lefschetz theorem is true for $Y$ even if $Y$ is not ample.
\end{cor}
\begin{proof}
By corollary \ref{cordef} there is a deformation $X'$ of $X$ preserving the Hodge class $[Y]$ such that this class is ample in $X'$. The cohomology rings of $X$ and $X'$ are isomorphic. Hence the hard Lefschetz theorem holds for $Y$ in $X$.
\end{proof}
Next we recall the classical version of the Lefschetz hyperplane section theorem.
\begin{thm}[LHT]Let $X\subset \pr^N$ be a smooth projective variety of dimension $n$. Take a smooth hyperplane section $Y:=X\cap \pr^{N-1}$. For any $k$, $0 \leq k<n-1$ the restriction map 
$$\h^k(X,\mathbb{Z})\rightarrow \h^k(Y,\mathbb{Z}).$$ 
is an isomorphism, and for $k=n-1$ it is injective.
 
\end{thm}
To study nef and big hypersurfaces in an irreducible holomorphic symplectic manifold we will need a version of the Lefschetz theorem about hyperplane section for semismall maps. Start with defining semismall maps. 
\begin{dfn}\label{semismall}
Let $f: X \to Y$ be a morphism of algebraic varieties. We call $f$ {\bf semismall} if for every subvariety $Z$ of $X$, the following inequality holds $$2\dim Z\leq \dim X+\dim f(Z).$$ 
\end{dfn}

\begin{thm}\label{lhtsemismal}\cite[Part~II, Sect~1.1]{GM}
Let $f: X \to \pr^N$ be a (not necessarily proper) semismall map of a non-singular purely $n$-dimensional algebraic variety into the complex projective space. Take a smooth hyperplane section $Y:=X\cap f^{-1}(\pr^{N-1})$. Then for any $0 \leq k<n-1$ the restriction map induces an isomorphism of the cohomology groups:
$$\h^k(X,\mathbb{Q})\xrightarrow{\sim}\h^k(Y,\mathbb{Q}).$$ 
\end{thm}
\section{Foliations and invariant subvarieties\label{s2}}
In this section we recall some definitions related to foliations and  preliminary results about the Zariski closure of the leaves. First, we define the foliations and leaves and also mention the Frobenius theorem.
\begin{dfn}Let $X$ be a smooth variety. A (singular) foliation is a saturated
subsheaf $F \subset T_X$ which is closed under the Lie bracket, i.e. $[F, F] \subset F$. The singularity locus $Sing(F)$ of $F$ is the subset of $X$ on which $T_X/F$ is not locally free, and it has codimension at least $2$ in X . A leaf of $F$ is the maximal connected injectively immersed complex analytic submanifold $L \subset X\setminus sing(F)$ such that $T_L =F|_L$.
\end{dfn}
A saturated subsheaf $F \subset T_X$ which does not necessarily satisfy the property $[F, F] \subset F$ is called a distribution. The property $[F, F] \subset F$ is needed for the existence of leafs. 
\begin{thm}[Frobenius]\label{thmfrob} Let $X$ be a smooth variety and $F\subset T_X$ a distribution on $X$. We say that $F\subset T_X$ is integrable if there exists a leaf (i.e. locally closed submanifold $L \subset X\setminus sing(F)$ such that $T_L =F|_L$) through every point of $X\setminus sing(F)$. The distribution $F$ is integrable if and only if $F$ is closed  under the Lie bracket, i.e. $[F, F] \subset F$. 
\end{thm}
\begin{proof}
See for example \cite[Book~I, Theorem~2.20]{Vo}.
\end{proof}
 \begin{dfn}
 If every leaf of a foliation is algebraic we call this foliation {\bf algebraically integrable}.
 \end{dfn}
 In this paper we study only foliations of rank one. It follows from theorem \ref{thmfrob} that all distributions of rank one are foliations (i.e. integrable).  
\begin{dfn}\label{inv}  Let $Y$ be  a closed smooth subvariety of $X$. One says it is {\bf invariant under the foliation}  $F$  or {\bf $F$--invariant} if $T_Y$ contains $F|_Y$.
\end{dfn}
The Zariski closure of a leaf through a point $x$ is the smallest invariant under $F$ subvariety containing this point. We denote it by $\overline{Leaf}^{Zar}(x,F)$.
We recall some results of Philippe Bonnet from the work \cite{Bo}. He states them for an affine variety $X$. Nevertheless, these statements for an affine $X$ obviously lead to the analogous statements for a projective. variety. Thus, let us reformulate them for a projective variety $X$.
\begin{thm}\cite[Theorem~1.3]{bo2} Let $X$ be a projective variety with a foliation $F$.
There is an integer $m$ such that $m$ is equal to the dimension of the Zariski closure of the leaf of $F$ through a very general \footnote{Outside of a countable union of proper closed subvarieties.} point $x\in X$. We call this integer $m$ the dimension of the Zariski closure of a generic leaf. Moreover, the dimension of the Zariski closure of the leaf through every point $x\in X$ is not greater than $m$.
\end{thm}
   
\begin{prop}\cite[Theorem~1.4]{bo2}\label{fibr} 
Let $X$ be a smooth projective variety of dimension $n$ with a foliation $F$. Assume that the Zariski closure of a very general leaf of $F$ has dimension $m<n$. Then there exists a rational map $X\dashrightarrow W$ with $F$-invariant fibers of dimension $m$ and a very general fiber of this map is the Zariski closure of a leaf of $F$. 
\end{prop}{}

In the work \cite{KCT} the authors proved the following  consequence of the Bogomolov-McQuillan theorem (\cite{BM}) and of the Reeb stability theorem. We formulate it for a foliation of rank one. 
\begin{thm}\cite[Theorem~2]{KCT}\label{BogMc}
Let $X$ be a smooth variety with a regular foliation $F\subset T_X$ of rank one. Assume that there is a curve $C\subset X$, such that $\deg F|_C>0$. Then all leaves of $F$ are rational curves.
\end{thm}
\section{The case of an ample hypersurface}\label{secample}
In this section we prove conjecture \ref{conja} for a smooth and ample hypersurface $Y$. 
\begin{thm}\label{ample}
Let $Y$ be a smooth and ample hypersurface in an irreducible holomorphic symplectic manifold $X$.  Then a generic leaf of the characteristic foliation is Zariski dense in $Y$. 
\end{thm}

In order to prove theorem \ref{ample} we assume the contrary. Let $Y \dasharrow B$ be a rational fibration such that its general fiber is invariant under the characteristic foliation $F$. Without loss of generality one can assume that $B$ is the projective line. We replace $Y\dashrightarrow B$ by the composition $Y\dashrightarrow B \dasharrow \pr^1$, where  $B \dasharrow \pr^1$ is a pencil of  hypersurfaces in $B$. Let $Z$ be a fiber of this rational fibration. 
\begin{dfn}
A subvariety $Z$ of codimension $k$ is called {\bf coisotropic} if the restriction of the symplectic form $\sigma$ to the tangent space to $Z$ at a general point has the smallest possible rank $n-k$ (if the rank of the restriction is smaller, then the form $\sigma$ is degenerate on $T_{X}$).
\end{dfn}
\begin{lem}\label{cois}
Let $(X,\sigma)$ be an irreducible holomorphic symplectic manifold and $Y$ a smooth hypersurface in $X$. Consider a (possibly singular) subvariety $Z$ of codimension 2 in $X$ contained in $Y$ (i.e. a hypersurface in $Y$). The following statements are equivalent \footnote{It is easy to see that the implication $2\Rightarrow 1$ is true for any codimension of $Z$.}\begin{enumerate}
    \item The variety $Z$ is invariant under the characteristic foliation;
    \item $Z$ is coisotropic  in $(X,\sigma)$ (the restriction of $\sigma$ to the smooth locus of $Z$ has the least possible rank $n-2$).
\end{enumerate}   \end{lem}
\begin{proof}
Let $z$ be a smooth point of $Z$. Consider the vector spaces $T_{Z,z}\subset T_{Y,z}\subset T_{X,z}$. \\
$\Longrightarrow$ Since $Z$ is $F_Y$-invariant, $T_{Z,z}$ contains $T_{Y,z}^{\perp}$. The line $T_{Y,z}^{\perp}$ is orthogonal to any vector of $T_{Z,z}$. Thus $\sigma|T_{Z,z}$ is degenerate. Since $T_{Z,z}$ has codimension $2$, it is coisotropic. 
\\$ \Longleftarrow$ Since $Z$ is coisotropic, $T_{Z,z}$ contains $T_{Z,z}^{\perp}$ and hence it contains $T_{Y,z}^{\perp}=F_{Y,z}$ \footnote{As one may notice, for $Z$ of higher codimension the first implication is wrong but the second implication holds.}.
\end{proof}
Theorem \ref{ample} obviously follows from the next result. 
\begin{prop}\label{coisample}
Let $Y$ be an ample smooth hypersurface in an irreducible holomorphic symplectic manifold $X$. Then $Y$ contains no coisotropic subvariety of codimension $2$ in $X$.
\end{prop}
In the rest of this section, we prove proposition \ref{coisample}. First we remark that being coisotropic is a cohomological property. In other words, a subvariety $Z$ of codimension $k$ (possibly singular) is coisotropic if and only if $[Z]\cup [\sigma^{n-k+1}]=0 \in \h^{2n+2}(X,\Co)$ (see \cite[lemma~1.4]{Vo3} for the details). Now we use the ampleness of $Y$ to apply the Lefschetz hyperplane theorem (LHT). It yields that there is a (not necessarily effective) divisor $D$, such that $[Z]=[D]\cdot [Y]$. 
\begin{lem}
Let $\alpha,\beta \in \ns(X)$. The class $\alpha \cup \beta$ is coisotropic if and only if $q(\alpha,\beta)=0$.
\end{lem}
\begin{proof}
If $Z$ is coisotropic $[Z]\cup [\sigma^{n-1}]=0$ and hence $[Z]\cup [\sigma^{n-1}]\cup[\bar{\sigma}^{n-1}]=\alpha \cup \beta \cup [\sigma]^{n-1}\cup [\bar{\sigma}]^{n-1}=0$.  Since $\alpha$ and $\beta$ of type $(1,1)$,  $q(\alpha, \beta)$ is proportional to $\alpha \cup \beta \cup [\sigma]^{n-1}\cup [\bar{\sigma}]^{n-1}$. Hence, if $Z$ is coisotropic, then $q(\alpha, \beta)=0$.
\end{proof} So, $[D]$ is $BBF$--orthogonal to $[Y]$. We show that it is impossible. 
\begin{lem}\label{null}
Let $\alpha, \beta \in \ns(X)$ and $q(\beta,\beta)>0$. Then the signs of $q(\alpha, \beta)q(\beta,\beta)^{n-1}$ and of $\alpha \beta^{2n-1}$ are the same.
\end{lem}
\begin{proof}
 
 The Fujiki formula says that there is a positive constant $c$ such that for every $\gamma\in \h^2(X,\mathbb{Z})$
$$q(\gamma,\gamma)^n=c\gamma^{2n}.$$
Let $k$ be an integer. Applying the Fujiki formula for $k\beta +\alpha$ we obtain the equality of polynomials in $k$:
\begin{equation}\label{polynom}
    c(k\beta+\alpha)^{2n}=(k^2q(\beta)+2q(\alpha,\beta)k+q(\alpha))^n.
\end{equation}
This equality of polynomials gives us the equality of the coefficients of the term of degree $2n-1$: 

$$2nc(\alpha \beta^{2n-1})=2nq(\alpha,\beta)q(\beta,\beta)^{n-1}.$$ 
\end{proof}
\begin{cor}\label{cornull}
Let $Y$ be an ample divisor and $D$ a divisor such that $[Z]= [Y] \cup [D]$ is an effective divisor. Then $q(Y,D)>0$. In particular, $Z$ is not coisotropic.
\end{cor}
\begin{proof}

Take $\alpha=[D]$ and $\beta=[Y]$. Since $Z$ is effective and $Y$ is ample we have $$[Y]^{2n-1}\cup [D]=[Y]^{2n-2}\cup [Z]>0$$ and hence $q(D,Y)>0$. 



\end{proof}
In the end of this section we make a conjecture generalizing proposition \ref{coisample}.
\begin{conj} Let $X$ be an IHS manifold of dimension $2n$ and $Y$ a smooth ample hypersurface in $X$. Then the hypersurface $Y$ contains no coisotropic subvariety except itself.
\end{conj}
If we apply the arguments we used to prove proposition \ref{coisample}, we can decompose the class of a coisotropic subvariety $Z$ as $[Z]=[Y]\cup\alpha$, where $\alpha \in \h^{2\dim Z-2}(X,\mathbb{Z})$. For $codim Z>2$ we meet the problem that the class $\alpha$ is not necessarily polynomial (i.e. is not contained in the image of the morphism $S^{\dim Z -1}\h^2(X,\mathbb{Z})\to \h^{2 \dim Z -2}(X,\mathbb{Z})$). And we can not use the estimations with BBF form to this class $\alpha$.
\section{The LHT for a nef and big hypersurface in IHS}\label{seclht}
In the previous sections we used the LHT for an ample hypersurface. In order to adapt this proof for a nef and big hypersurface we show that the LHT holds for such hypersurfaces. The aim of this section is to prove theorem  \ref{lht}.

First we recall that the LHT (with rational coefficients) follows from the Kodaira–Akizuki–Nakano vanishing. 

\begin{lem}\label{LHTKAN}Let $X$ be a smooth variety and $\mathcal{L}$ be an effective line bundle on $X$. Consider a smooth hypersurface $Y\subset X$ such that $\ob_X(Y)\cong \mathcal{L}$. Assume that Kodaira-Akizuki-Nakano vanishing holds for $\mathcal{L}$. This means the following: $$\h^q(X, \Omega^p_X\otimes \mathcal{L}^*)=0\; \textrm{for}\; \; p+q<\dim X.$$ Then  $\h^i(X,\mathbb{Q} )\cong \h^i(Y,\mathbb{Q})$ for $i<\dim X$.
\end{lem}
\begin{proof}
See \cite[Chapter~13.3]{Vo}.
\end{proof}
Thus, to prove the LHT it is enough to prove the Kodaira–Akizuki–Nakano vanishing. Let us formulate it.
\begin{prop}
Let $\mathcal{L}$ be a nef and big line bundle on an irreducible holomorphic symplectic manifold $X$. Then the Kodaira–Akizuki–Nakano vanishing (we will write it as the KAN for shortness) holds for $\mathcal{L}$:
$$ \h^q(X, \Omega^p_X\otimes \mathcal{L}^*)=0\; \textrm{for}\; \; p+q<\dim X.$$
\end{prop}

Remark that the KAN and the LHT are generally false for a nef and big line bundle. 
\begin{rmk}
In \cite[Remark~4.3.3]{Laz} R. Lazarsfeld gives an example, where the KAN (or the LHT for a generic section) is not true for a nef and big line bundle. This line bundle is a pull-back of $\ob_{\pr^3}(1)$ on the projective space $\pr^3$ to the blowing-up $\Bl_P \pr^3$ of this space at a point $P\in \pr^3$.
\end{rmk}
The clue in the holomorphic symplectic case is that the morphism induced by a nef and big line bundle is of a special kind. Namely, these morphisms are semismall. Let us recall the definition of a {\it lef} line bundle from \cite{semilef}.

\begin{dfn} A line bundle $\mathcal{L}$ is called lef if some its power $\mathcal{L}^{\otimes k}$ is generated by global section and induces a semismall morphism to projective space.
\end{dfn}
E. Esnault and E. Viehweg proved the KAN vanishing for lef line bundles.
\begin{thm}{\cite[Theorem~2.4]{EV}}\label{esn}
Let $L$ be a lef line bundle on a smooth variety $X$. Then $$ \h^q(X, \Omega^p_X\otimes \mathcal{L}^*)=0\; \textrm{for}\; \; p+q<\dim X.$$
\end{thm}
We are going to prove that a nef and big line bundle $\mathcal{L}$ on a holomorphic symplectic manifold $X$ is lef. By the Kawamata-Shokurov base-point-free theorem (see \cite[Chapter~10]{CKM}) the line bundle $\mathcal{L}^{\otimes k}$ for $k>>0$ is generated by global sections.\\
\begin{lem}The morphism $\phi: X \xrightarrow{|kY|} \pr^N$ is semismall.
\end{lem}
\begin{proof}
Let $Z$ be a closed subvariety of $X$ of codimension $r$. By corollary \ref{hard}, $[Z]\cup [Y]^{2n-2r}\neq 0 \in \h^{4n-2r}(X,\mathbb{Z})$. Hence the section of $\phi(Z)$ by a linear space of codimension $2n-2r$ in $\pr^N$ is not empty. Finally we obtain,  $$\dim \phi(Z)\geq 2n-2r=2(2n-r)-2n=2\dim Z -\dim X.$$
Thus, the line bundle $\mathcal{L}$ is lef. By theorem \ref{esn} the KAN holds for $\mathcal{L}$. By lemma \ref{LHTKAN} the LHT with rational coefficients is true for $Y$.
\end{proof}

\section{The case of a nef and big hypersurface}\label{secnef}
In this section we prove theorem \ref{tnef}. Let us recall it.\\

\noindent
\textbf{Theorem \ref{tnef}}
Let $Y$ be a smooth nef and big hypersurface in an irreducible holomorphic symplectic manifold $X$. Then a generic leaf of the characteristic foliation of $Y$ is Zariski dense in the hypersurface $Y$. \\

As it was shown in section \ref{secample} it is enough to prove  the following. 
\begin{prop}
 Let $Y$ be a smooth nef and big hypersurface in an irreducible holomorphic symplectic  manifold $X$. Then $Y$ can not be covered by a family of coisotropic subvarieties of codimension 2 in $X$.
\end{prop}
The problem is that even though the LHT holds, it is not enough to apply the arguments of section \ref{secample} to a nef and big hypersurface $Y$. Indeed when $Y$ is not ample one may have $Y^{2n-1}\cdot D=0$ (cf. proof of Corollary \ref{cornull}). 
But using that the family of coisotropic subvarieties covers $Y$ we can show that $D^2 \cdot Y^{2n-2} \geq 0$. Indeed, let $Z_1,Z_2$ be two distinct members of this family. Then $Z_1 \cap Z_2$ is an effective cycle of codimension $4$ in $X$. Thus, the intersection number $Z_1\cdot Z_2\cdot Y^{2n-4}=D^2 \cdot Y^{2n-2}$ is not negative. This observation contradicts to the following. 
\begin{lem}
In the assumption of lemma \ref{null} if $q(\beta,\beta)>0$, then $\alpha^{2}\cup\beta^{2n-2}<0$. 
\end{lem}
\begin{proof}
The signature of the restriction of the BBF form  to $\h^{1,1}(X)$ is $(1,h^{1,1}(X)-1)$. Hence, the BBF square of $\alpha$ is negative. Considering the equality of polynomials (\ref{polynom}) at the terms of degree $2n-2$ we obtain:
$$c \cdot\frac{2n(2n-1)}{2}\cdot\alpha^2\cup\beta^{2n-2}=\frac{n(n-1)}{2}q(\alpha,\alpha)q(\beta,\beta)^{n-1}.$$
Since $q(\alpha,\alpha)<0$ and $q(\beta,\beta)>0$, the intersection number $\alpha^2\beta^{2n-2}$ is negative. 
\end{proof}
Thus finishing the proof
\section{The case of a non numerically effective hypersurface }\label{nenef}
In this section we prove that a generic leaf of the characteristic foliation of $Y$ is also dense in $Y$ if $Y$ is not nef, under some additional condition on $Y$. Moreover, using some remarks by Jorge Vitorio Pereira, we show that a non-nef divisor with positive square must be singular. Thus finish the prove of the third case of conjecture \ref{conja}.\\
 A divisor with positive the Beauville-Bogomolov-Fujiki square is not necessarily nef. But it can be transformed to a nef divisor by a birational modification of $X$. The following proposition is a consequence of theorem 1.2 and  lemma 2.2 in \cite{MatZh}.
\begin{prop}\label{mz}\cite{MatZh}
 Let $X$ be an irreducible holomorphic symplectic manifold and $Y$ an irreducible hypersurface with $q(Y,Y)>0$. Then there is an irreducible holomorphic symplectic manifold $X'$ with birational map $\psi: X' \dasharrow X$, such that the divisor $Y'=\psi^*Y$ is nef.
\end{prop}
Let us recall an important fact on the birational transformation of IHS manifolds.
\begin{lem}\cite[page~420]{Ka}
Let $\psi:X\dashrightarrow X'$ be a birational transformation of IHS manifolds, then $\psi$ is an isomorphism in codimension $1$. 
\end{lem}

It is easy to see that if a generic leaf of the characteristic foliation is dense in $Y'$, then the same is true for $Y$. Theorem \ref{tnef} (Campana's conjecture for a nef and big hypersurface) has the following corollary.
\begin{cor}
In the assumptions of proposition \ref{mz}, if $Y'$ is smooth, then a generic leaf of the characteristic foliation is dense in $Y$. 
\end{cor}
In a recent conversation Jorge Vitorio Pereira proposed the following solution of the non-nef case. This result is also very interesting out of the context of conjecture \ref{conja}.\\

\noindent 
\begin{thm}Let $X$ be a holomorphic symplectic manifold and $Y$ a smooth non-nef hypersurface in $X$. Then $Y$ is uniruled.
\end{thm}
\begin{proof}
First, we notice that the line bundle $\ob_Y(Y)$ is not nef. Indeed, $\ob_X(Y)$ is not nef. Hence, there exists a curve $C\subset X$, such that the line bundle $\ob_C(Y)$ has negative degree. Since $Y$ intersects $C$ negatively, $C$ is contained in $Y$.\\

 Recall that the characteristic foliation $F$ is isomorphic to $\ob_Y(-Y)$. Thus the restriction of $F$ to the curve $C$ has positive degree. That allows us apply theorem \ref{BogMc} to the variety $Y$ and the foliation $F$. We obtain that all leaves of the characteristic foliation on $Y$ are rational curves. Hence, $Y$ is uniruled. 


\end{proof}
\noindent 
\textbf{Theorem \ref{per}} If $Y$ is smooth and uniruled then $q(Y,Y)<0$.
\begin{proof}
The proof of this statement is implicitly contained in the work \cite{Bo}. Let us summarize it. By \cite[theorem~4.5]{Bo} an "exceptional" hypersurface has negative BBF square. We need to show that a smooth uniruled divisor is "exceptional" according to definition of \cite{Bo}. S. Boucksom defines a "modified nef divisor" (the exact definitions are not important for us). For an irreducible divisor being exceptional is equivalent to being not "modified nef" \cite[Definition~3.10]{Bo}. What is important for us is that the restriction of a modified nef divisor to any prime divisor is pseff \cite[Proposition~2.4]{Bo}. \\

Note that $\ob_Y(Y)$ is not pseff. Indeed, $\omega_Y\cong \ob_Y(Y)$. Canonical bundle of a uniruled variety is not pseff. Hence, $Y$ is exceptional and we can apply \cite[theorem~4.5]{Bo} to a smooth uniruled hypersurface $Y$. 
\end{proof}
\begin{cor}\label{corcor}A non-nef hypersurface $Y$ with non-negative Beauville-Bogomolov square is singular.
\end{cor}
That completes the proof of the third case of conjecture \ref{conja}. 

\bibliography{cf}
\end{document}